\documentclass[a4paper, 12pt]{amsart}
\usepackage{amsmath,amsthm,amssymb}
\usepackage[T1]{fontenc}
\usepackage[margin=3.1cm]{geometry}
\usepackage[numbered]{bookmark}
\usepackage{changepage} 
\usepackage[utf8]{inputenc}
\usepackage{graphics, graphicx, caption}
\usepackage{amsmath, amsfonts, amsthm, amssymb}
\usepackage{mathrsfs, mathtools, mathabx}
\usepackage{verbatim} 
\usepackage{hyperref}
\usepackage{cleveref}
\usepackage{enumerate}
\usepackage{xcolor}
\usepackage{cite}

\newtheorem{theorem}{Theorem}[section]

\newtheorem{lemma}[theorem]{Lemma}
\newtheorem{proposition}[theorem]{Proposition}

\newcommand{\Z}{\mathbb{Z}} 
\newcommand{\N}{\mathbb{N}}

\newcommand{\R}{\mathbb{R}}

\newcommand{\la}{\lambda}

\author[Berk]{Przemysław Berk}
\address{Faculty of mathematics and computer science, Nicolaus Copernicus University, ul. Chopina 12/18, 87-100 Toru\'n, Poland}
\email{zimowy@mat.umk.pl}

\author[Kotlewski]{Łukasz Kotlewski}
\address{Faculty of mathematics and computer science, Nicolaus Copernicus University, ul. Chopina 12/18, 87-100 Toru\'n, Poland}
\email{kotlewskilukasz@mat.umk.pl}

\title[Recurrence properties of certain non-integrable cocycles]{Coexistence of two
contrasting recurrence properties of certain non-integrable cocycles}

\begin{document}
\maketitle

\begin{abstract}
    We study the recurrence properties of certain skew products over symmetric interval exchange transformations, including rotations, with cocycles of the form $f(x)=-\frac{1}{x^a}+\frac{1}{(1-x)^a}$, where $a>1$. We prove that typically, such systems are dissipative. However, at the same time they are \emph{topologically recurrent}, i.e. for every open rectangle $A\subset[0,1)\times \R$, there exists an infinite sequence $(q_n)_{n=1}^{\infty}$ such that $T^{q_n}_f(A)\cap A\neq\emptyset$.
\end{abstract}

\section{Introduction}

Let $T:(X,\mu)\to(X,\mu)$ be an ergodic probability-preserving transformation and let
$\varphi:X\to\mathbb R$ be a measurable function. The associated skew product is the infinite-measure preserving transformation
\[
T_\varphi(x,y)=(Tx,\,y+\varphi(x)),\qquad (x,y)\in X\times\mathbb R,
\]
which preserves $\mu\otimes \mathrm{Leb}_{\mathbb R}$.
Iterates satisfy
$T_\varphi^n(x,y)=(T^n x,\,y+S_n\varphi(x))$, where $S_n\varphi$ are the Birkhoff sums of $\varphi$ with respect to $T$.
Such extensions can be interpreted as deterministic analogues of random walks
driven by the base dynamics and have long been studied in infinite ergodic theory
and in the theory of cocycles; see, for instance, \cite{Aaronson1997,Schmidt1977,Schmidt2006}.

A central dichotomy for skew products is whether the system is \emph{recurrent or dissipative}. In the integrable setting, the picture is particularly rigid:
if $T$ is ergodic and $\varphi\in L^1(\mu)$, recurrence of the cocycle (equivalently, conservativity of
$T_\varphi$) is governed by the vanishing of the mean, and the recurrence of almost every point follows from the result of Atkinson \cite{atkinson_recurrence_1976}, see also Schmidt \cite{Schmidt2006}. On the other hand, if $\varphi$ is of non-zero mean, then by the classical Birkhoff Ergodic Theorem, almost every point $x\in X$ escapes, i.e. for every $L>0$ at most finitely many iterations of $T_{\varphi}$ on $x$ belong to the set $X\times[-L,L]$.
Beyond integrability, however, even qualitative behavior of Birkhoff sums may change
drastically: large excursions can be forced by rare visits of the base orbit to regions
where $\varphi$ has large values, and the classical $L^1$-based tools, such as essential values, are no longer available. Clearly, when a non-integrable function can be estimated from below by $L^1$ functions with positive mean or from above by functions with negative mean, then one obtains the escape of points in a similar manner as when one considers functions with zero mean. Hence, to get new interesting examples, we need to carefully choose the cocycles with additional properties, which do not allow for this kind of exploitation.

In this work we focus on a concrete and, from the cocycle perspective, genuinely
\emph{non-integrable} regime. The base transformation is an interval exchange
transformation (IET) $T:[0,1)\to[0,1)$, i.e.\ a piecewise translation with finitely many
discontinuities. For irreducible combinatorics, typical IETs are minimal
and uniquely ergodic \cite{Masur1982,Veech1982}, and renormalization by \emph{Rauzy--Veech
induction} provides a powerful tool for Birkhoff-sum estimates. 

Skew products over rotations and IETs with \emph{singular cocycles} have attracted
substantial interest, in part because they arise as Poincar\'e maps for extensions of
area-preserving flows on surfaces. In the rotation case, Fayad and Lema\'nczyk proved
ergodicity for a full-measure set of angles for cylindrical transformations with
logarithmic singularities (under a zero-mean assumption) \cite{FayadLemanczyk2006}.
For IETs, recent progress has produced explicit ergodic examples for cocycles with
(logarithmic) singularities over typical symmetric IETs \cite{BerkTrujilloUlcigrai2025},
and more generally a method, involving Borel--Cantelli-type arguments, has been
developed to prove ergodicity for classes of antisymmetric cocycles with singularities
over symmetric IETs \cite{BerkFraczekTrujillo2024}. These advances operate in settings
where the cocycle is still (at least) integrable. To our knowledge, no  results exist in this spirit about cocycles over IETs (or even just rotations), with the cocycle being non-integrable. It is important to mention that non-integrable cocycles may as well appear when studying the Poincar\'e sections of extensions of area-preserving flows on surfaces, when the extension is being made via a non-integrable energy function.

Hence, our goal here is the following: we exhibit and analyze an \emph{escape} phenomenon
for an explicit family of \emph{non-integrable} cocycles. Namely, for $a>1$ we consider
\[
f(x)=-\frac{1}{x^a}+\frac{1}{(1-x)^a},
\]
which has power singularities at $0$ and $1$ and since $a>1$, we have $f\notin L^1([0,1))$.
We study the skew product $T_f$ over \emph{symmetric} IETs (whose permutation reverses
the order of exchanged intervals), a class which has an additional antisymmetry
that interacts naturally with the involution $\mathcal I(x)=1-x$ (see Sec. \ref{sec: iets} for details). 
Our main result is as follows.
\begin{theorem}\label{thm: main}
    Let $f:[0,1)\to \R$ be given by $f(x)=-\frac{1}{x^{a}}+\frac{1}{(1-x)^{a}}$, with $a>1$. For almost every symmetric interval exchange transformation $T$ on $[0,1)$, the skew product $T_{f}$ on $[0,1)\times\R$ given by
    \[
    T_{f}(x,y)=(T(x),y+f(x))
    \]
    has the following properties:
    \begin{itemize}
        \item[A)] for a.e. point $(x,y)\in [0,1)\times \R$, and every $L>0$ there exists $N=N(x,y,L)$ such that for every $n\ge N$ it holds that $T_{f}^n(x,y)\notin[0,1)\times [-L,L]$;
        \item[B)] for every open rectangle $D:=I_1\times I_2\subset [0,1)\times \R$, there exists an infinite increasing sequence $(q_n)_{n=1}^{\infty}$ such that $T^{q_n}_{f}(D)\cap D\neq\emptyset$.
    \end{itemize}
    In particular, by Property \textup{A}, $T_f$ is not ergodic with respect to $Leb_{[0,1)}\otimes Leb_{\R}$.\end{theorem}

Loosely speaking, we show that in the considered family of skew products, our systems are indeed typically dissipative, that is almost every point in $[0,1)\times\R$, with respect to $Leb_{[0,1)}\otimes Leb_{\R}$ escapes.
At the same time, and in sharp contrast with the measure-theoretic escape,
we prove a complementary form of \emph{topological return}: every open rectangle
$D\subset [0,1)\times\mathbb R$ intersects infinitely many of its images along a subsequence
of iterates. This shows the possibility of coexistence of two contrasting properties, when dealing with non-integrable cocycles, a phenomenon not observed in the integrable framework in the case of minimal bases of skew products and continuous cocycles. To our knowledge, it is the only known result of this kind. However, from a different point of view, it is worth mentioning that the limit behavior of Birkhoff sums of cocycles studied in this article, in case of rotations, were studied extensively by Sinai and Ulcigrai in \cite{SinaiUlcigrai2008LimitTheorem} and more recently by Auer and Schindler in \cite{AuerSchindler2025TrimmedPower}.

We note here that assuming minimality of $T$ and at least piecewise continuity of $\varphi$ is a natural assumption when combining measure-theoretic and topological regime. Otherwise, the interplay between topological and measure-theoretic phenomena can trivialize. For example, if $X=\{0,1\}^\Z$, $\mu\sim \operatorname{Ber}(p)$ with $p\in(0,1)\setminus\{1/2\}$ and $\varphi=\chi_{[0]}-\chi_{[1]}$, where $[0]$ and $[1]$ are cylinders corresponding to fixed $0$-th coordinate, then in every open set in $X$ one can find a periodic sequence which has the same number of 0's and 1's and thus the topological recurrence for the shift holds easily, while $\int \varphi\ d\mu\neq 0$ and thus, the system is measure-theoretically dissipative. 

We emphasize that Theorem \ref{thm: main} does not hold for the function $f=-\frac{1}{x}+\frac{1}{1-x}$. More precisely, the property A cannot be proven the same way we are proving it for the considered class of functions. The reason is that the key Borel-Cantelli-like argument does not work. Interestingly, the property B can be shown for $f=-\frac{1}{x}+\frac{1}{1-x}$ and the proof is identical to ours. This poses a question, whether skew products given by this cocycle are recurrent and if they are, whether they are ergodic.

\smallskip
\noindent\textbf{Outline.}
Section~2 recalls the IET framework and fixes Diophantine conditions required to obtain the main results of this article. In Section~3 we derive lower bounds for Birkhoff sums of the
derivative of $f$, which yield quantitative control on the size of level sets where
$S_n f$ is small. Section~4 applies a Borel--Cantelli argument to deduce the almost-sure
escape from bounded vertical strips. Finally, Section~5 exploits the symmetry of the base
and the antisymmetry of $f$ to produce the subsequence of iterates yielding intersections
of rectangles.

    \section{Interval exchange transformations and the necessary Diophantine conditions}\label{sec: iets}
    Let us recall that an \emph{interval exchange transformation} (or \emph{IET} for short) $T$ on $[0,1)$ is any automorphism of $[0,1)$ such that there exists a partition of $[0,1)$ into finitely many intervals $I_{\alpha}$, for which $T|_{I_{\alpha}}$ is a translation. Such transformation clearly preserves the Lebesgue measure $Leb_{[0,1)}$ on $[0,1)$. 

    Every IET $T$ is fully determined by a couple of parameters:
    \begin{itemize}
        \item a permutation $\pi\in S(d)$ of the set $\{1,\ldots,d\}$, which governs the way the intervals are being exchanged;
        \item the length vector $\lambda=(\lambda_{p})_{1\le p\le d}\in\Lambda^{d}$, such that $|I_{p}|=\lambda_{p}$\footnote{Here $\Lambda^{d}$ denotes the standard unit simplex in $\R^{d}$}.
    \end{itemize}
    Hence the space of interval exchange transformations can be fully parametrized by the set $S^d\times \Lambda^{d}$. Whenever we say that a property holds for almost every IET, we mean that it holds for $Leb$-a.e parameter $\lambda\in\Lambda^d$. We also denote by $\partial I_p$, the left-hand side endpoint of the interval $I_p$ for every $1\le p\le d$.

      Throughout this paper, we will assume that $T$ satisfies the Keane condition, that is
      $T^j(\partial I_p)=\partial I_{p'}$ for some $j\in\Z$ and $1\le p,p'\le d$ implies $j=1$ and $p'=1$. Moreover, we are interested in IETs whose associated combinatorial data $\pi$ is a symmetric permutation, that is
 \[
     \pi(p)=d+1-p\quad\text{for every }1\le p \le d.
     \]
     It is worth mentioning that the importance of this permutation comes from the fact that if $\mathcal I:[0,1)\to [0,1)$ is given by $\mathcal I(x)=1-x$ (defined everywhere except $0$), then $T^{-1}\circ \mathcal I=\mathcal I\circ T$ (which holds always outside the orbits of discontinuities of $T$). The Lemma \ref{lem:zeroes}, used in the proof of Property B in \ref{thm: main} is a consequence of this property.

For a typical interval exchange transformation we have the following proposition, whose statement contains also the notation used in the proofs of main results.
\begin{proposition}\label{prop: properties}
    For almost every symmetric IET $T$ of $d\ge 2$ intervals on $[0,1)$, $T$ satisfies the Keane condition and there exists an infinite sequence of intervals $I^{(n)}=[0,|I^{(n)}|)$ with $\lim_{n\to\infty}|I^{(n)}|=0$ such that: 
    \begin{itemize}
        \item[1)] the first return map $T^{(n)}$ to $I^{(n)}$ is a symmetric IET of $d$ intervals denoted $(I^{(n)}_{p})_{1\le p\le d}$ with the length vector $\lambda^{(n)}=(\lambda^{(n)}_{p})_{1\le p\le d}$;
        \item[2)] there exists a constant $0<\epsilon_2<10^{-6}$ such that 
        \[
        \lambda_d>(1-\epsilon_2)|\lambda|,\ \text{where}\ |\la|:=\sum_{1\le p\le d}\la_p;
        \]
        \item[3)] for every $1\le p\le d$, the set  $\Xi_{p}^{(n)}:=
        \bigsqcup_{i=0}^{h^{(n)}_{p}-1} T^{i}I^{(n)}_{p}$ is a Rokhlin tower of intervals, where $h^{(n)}_{p}$ is the first return time of points from $I^{(n)}_{p}$ to $I^{(n)}$;
        \item[4)] for every $1\le p\le d$ and $0\le i\le h^{(n)}_{p}-1$, $T$ acts via translation on the interval $T^i(I^{(n)}_{p})$;
        \item[5)] there exists a constant $C_5>0$ such that 
        \[
        \max_{1\le p,p'\le d}\frac{h^{(n)}_{p}}{h^{(n)}_{p'}}\le C_5,
        \]
        \item[6)] there exists a constant $C_6>0$ such that if $q_n:=\max_{1\le p\le d}h^{(n)}_{p}$, then
        \[
        \lim_{n\to\infty}\frac{\log q_{n}}{n}=C_6;
        \]
         \item[7)] there exists a constant $C_7>0$ such that every orbit $(T^{i}x)_{i=0}^{h^{(n)}_{p}-1}$, where $1\le p\le d$ and $x\in I^{(n)}_p$,  divides $[0,1)$ into intervals of length at most $\frac{C_7}{q_n}$.
    \end{itemize}
\end{proposition}
The proof of the above proposition follows from very classical properties of the so-called \emph{Kontsevich-Zorich cocycle} on the moduli space of translation surfaces, namely the fact that it is ergodic with respect to a measure equivalent to the Lebesgue measure and the fact that it is integrable. Such techniques have been used a number of times throughout many articles on the subject, thus we omit the proof. However, we encourage the reader to compare the Proposition \ref{prop: properties} with the Proposition 5.1 in \cite{BerkFraczekTrujillo2024}, since it has a very similar statement and a detailed proof.

It is worth mentioning that the Proposition \ref{prop: properties} holds also for typical rotations. Then the sequence $q_n$ is simply the sequence of partial denominators and the Proposition \ref{prop: properties} follows easily from the ergodic properties of the classical Gauss map. Since rotations are just IETs of two intervals, our result in particular holds for a typical rotation as well.

\section{Lower bounds for the derivative.} We now show the required bounds for the Birkhoff sums of the derivative of function $f$. They will be useful in proving Property A from Theorem \ref{thm: main}. Thus, let $f=-\frac{1}{x^{a}}+\frac{1}{(1-x)^{a}}$ for some $a>1$. 

Fix an IET $T=(\pi,\lambda)$ which satisfies the conclusion of Proposition \ref{prop: properties} and let $(q_n)_{n\in\N}$ defined in the point 6). For every function $F:[0,1)\to\R$ and $n\in\Z$, we denote by $S_n(F)(\cdot)$ the $n$-th Birkhoff sum, i.e.
\[
S_n(F)(x):=\begin{cases}
    \sum_{k=0}^{n-1}F(T^k x)&\text{ if }n>0;\\
    0&\text{ if }n=0;\\
    -\sum_{k=n}^{-1}F(T^k x)&\text{ if }n<0.
\end{cases}
\]
We have the following result.
\begin{lemma}\label{lem:lowerbounds}
Let $r\in \N$ and let $n$ be such that $2q_n\le r<2q_{n+1}$. Then there exists a constant $c>0$, such that 
\[
S_{r}(f')(x)\ge cr q_n^{a},
\]
for every $x\in\mathbb [0,1)$ not in the orbit of singularity.
\end{lemma}

\begin{proof}
Throughout the proof we use the fact that $f'$ is a positive function. We first start by proving that for any $n\in\N$, $p\in\{1,\ldots,d\}$ and $x\in I^{(n)}_{p}$ we have that 
it holds that
\begin{equation}\label{eq: speciallowerbounds}
    S_{h^{(n)}_{p}}(f')(x)\ge c'h^{(n)}_{p}q_{n}^a
\end{equation}
for some constant $c'>0$.

Consider first the function $g(x)=\frac{1}{x^{1+a}}$. In view of 7) in Proposition \ref{prop: properties}, for every $x\in\mathbb [0,1)$ we have
\begin{equation*}
S_{h^{(n)}_{p}}(g)(x)=\sum_{j=0}^{h^{(n)}_{p}-1}g(T^j(x))\ge
\sum_{j=1}^{h^{(n)}_{p}}\left(\frac{2jC_7}{q_n}\right)^{-1-a}=({q_n})^{1+a}\sum_{j=1}^{h^{(n)}_{p}}\left({2jC_7}\right)^{-1-a}\ge \tilde ch^{(n)}_{p}q_n^{a},
\end{equation*}
where $\tilde c:=(2^{1+a}C_7^{1+a})^{-1}$. 
Analogously, we prove that if $h(x)=\frac{1}{(1-x)^{1+a}}$, then
\[
S_{h^{(n)}_{p}}(h)(x)\ge \tilde ch^{(n)}_{p}q_n^{a}.
\]
Since $f'=\frac{1}{a}(g+h)$, this shows \eqref{eq: speciallowerbounds} for $c'=2\tilde ca^{-1}$.

Assume now that $x\in[0,1)$, $r\ge 2q_1$ and let $n$ be such that $2q_n\le r<2q_{n+1}$. For any $1\le p\le d$, let $b_{p}$ be the number of times the orbit $\mathcal O(x,r):=\{T^jx\mid j=0,\ldots, r-1\}$ passes through the entire tower $\Xi_{p}^{(n)}$. Since $r\ge 2q_n$, and $q_n$ is the maximal height of towers of index $n$, at least for one index $1\le p \le d$, the number $b_{p}$ is positive. In particular
\[
\sum_{1\le p\le d}b_{p}h^{(n)}_{p}\le r <2q_n+\sum_{1\le p\le d}b_{p}h^{(n)}_{p}\le 3C_5\sum_{1\le p\le d}b_{p}h^{(n)}_{p},
\]
where the last inequality follows from 5) in Proposition \ref{prop: properties}.
Then by \eqref{eq: speciallowerbounds}, for $c:=(3C_5)^{-1}c'$ we get
\[
S_r(f')(x)\ge  \sum_{1\le p\le d}b_{p}c'h^{(n)}_{p} q_{n}^{a}\ge crq_n^a,
\]
This finishes the proof of Lemma \ref{lem:lowerbounds}.\end{proof}

\section{Escaping points.}
In this section we show the Property A of Theorem \ref{thm: main}. The idea to prove this part is to use the First Borel-Cantelli Lemma to show that points in general do not come back.

\begin{proof}[Proof of Property A in Theorem \ref{thm: main}.] {Let $T$ be a symmetric IET which satisfies conclusion of Proposition \ref{prop: properties}}. Let $r\ge 2q_1$ and $n\in\N$ such that $2q_n\le r<2q_{n+1}$. Note that, by Keane condition, the set $\bigcup_{1\le p\le d}\mathcal O(T^{-r}\partial I_{p},r)$ divides $[0,1)$ into $(d-1) r+1$ intervals of continuity of $S_r(f)$. Moreover, $f$ is strictly increasing, hence $S_r(f)$ has at most $(d -1)r+1$ zeroes. 

Fix $L>0$ and denote 
\[
A_r:=\{x\in [0,1) \mid |S_r(f)(x)|\le L\}.
\]
Let $z$ be one of the zeroes of $S_r(f)$. In view of Lemma \ref{lem:lowerbounds}, the interval $I_z$ which contains $z$, such that for every $x\in I_z$ we have $|S_r(f)(x)|\leq L$, has length $|I_z|\le \frac{2L}{crq_n^{a}}$. Since there are at most $(d -1)r+1$ zeroes of $S_r(f)$, we get that
\begin{equation}\label{eq:measofrecurrence}
Leb(A_r)\le \left((d-1) r+1\right)\cdot \frac{2L}{crq_n^{a}}\le \frac{4d L}{cq_n^{a}},
\end{equation}
with constant $c>0$ independent of $r$.

To finish the proof of Property A, it is enough to show that a.e. point $(x,y)$ visits the set $[0,1)\times[-L,L]$ via $T_f$ at most finitely many times. By Borel-Cantelli Lemma, it is enough to show that 
\begin{equation}\label{eq:convergecetoshow}
\sum_{r=1}^\infty Leb(A_r)<\infty.
\end{equation}

Note that by 6) in Proposition \ref{prop: properties}, for any given $\epsilon>0$ and $n$ large enough, we have for some $D>0$ that 
\begin{equation}\label{eq:q_n+1_vs_q_n}
    q_{n+1}\le e^{(C_6+\epsilon)(n+1)}=e^{C_6+\epsilon}
    \cdot (e^{(C_6-\epsilon)n})^{\frac{C_6+\epsilon}{C_6-\epsilon}}\le Dq_n^{1+\frac{(a-1)}{2}},
\end{equation}
which, by enlarging $D$ if necessary, may be assumed to hold for every $n\in\N$.
Thus, in view of \eqref{eq:measofrecurrence} and \eqref{eq:q_n+1_vs_q_n}, we get that 
\[
\begin{split}
&\sum_{r=2q_1}^\infty Leb(A_r)=\sum_{n=1}^{\infty}\sum_{r=2q_n}^{2q_{n+1}-1}Leb(A_r)\\
&\le \sum_{n=1}^{\infty}\sum_{r=0}^{2q_{n+1}-1}Leb(A_r)\le \sum_{n=1}^{\infty} 2q_{n+1}\cdot\frac{4d L}{cq_n^{a}}\le \frac{8d D L}{c}\sum_{n=1}^{\infty}{\frac{1}{q_n^{{(a-1)/2}}}}<\infty,
\end{split}
\]
where the last inequality follows again from 6) in Proposition \ref{prop: properties}. This shows \eqref{eq:convergecetoshow} and thus finishes the proof of Property A in Theorem \ref{thm: main}.
    
\end{proof}

\section{Intersecting rectangles}
In this section we show the property B of Theorem \ref{thm: main}. We use the antisymmetry property used by the first author in articles with Frączek, Trujillo, Ulcigrai and Wu (see \cite{BerkTrujilloUlcigrai2025}, as well as \cite{berk_ergodic_2024}, \cite{berk_ergodicity_2024} and \cite{BerkFraczekTrujillo2024}). Namely, we have the following lemma.
\begin{lemma}\label{lem:zeroes}
    Let $T$ be a symmetric IET satisfying the conclusion of Proposition \ref{prop: properties}. Moreover, let $f=-\frac{1}{x^a}+\frac{1}{(1-x)^a}$ {with $a>0$}. Then for every $n\in\N$, if $h_n:=h_d^{(n)}$, there exists a point $z_n\in\Xi_d^{(n)}$ such that 
    \begin{equation}\label{eq: zero}
    \lim_{n\to\infty} S_{2h_n}(f)(T^{-h_n}z_n)=0
    \end{equation}
    and
    \begin{equation}\label{eq: orbitfar}
    \textup{dist}\left(\mathcal O(T^{-h_n}z_n,2h_n),\{\partial I_p\mid 1\le p\le d\}\cup\{1\}\right)\ge \frac{1}{10}\la_d^{(n)}.
    \end{equation}
    Additionally, for every $n\in\N$, the point $z_n$ is a midpoint of the interval $T^{j_0}I^{(n)}_{d}$ for some $j_0\in\{0,\ldots,h_n-1\}$ and 
    \begin{equation}\label{eq: thesamelevel}
    |T^{-h_n}z_n-z_n|=|T^{h_n}z_n-z_n|\le \frac{1}{10}\la_d^{(n)}
    \end{equation}
\end{lemma}
\begin{proof}
    The position of the point $z_n$ follows from Lemma 3.14 in \cite{BerkTrujilloUlcigrai2025} and the property \eqref{eq: zero} follows analogously to the proof of Lemma 3.13 in \cite{BerkTrujilloUlcigrai2025} (one needs to replace the functions odd with respect to the exchanged intervals by the functions which are odd with respect to the whole domain, cf. Lemma 5.2 in \cite{berk_ergodicity_2024}). Then \eqref{eq: orbitfar} and \eqref{eq: thesamelevel} follow from the form of $z_n$ and 2) in Proposition \ref{prop: properties}.
\end{proof}

\begin{proof}[Proof of Property B in Theorem \ref{thm: main}.]
It is enough to show that for every open interval $J$ in $[0,1)$, for every $\epsilon$ there exists $n$ such that there exist a point $x\in J$ satisfying $|S_{2h_n}(f)(x)|<\epsilon$ and $T^{2h_n}(x)\in J$.

Thus, we will show that for a fixed open interval $J=(a,b)\subset [0,1)$ and $n$ sufficiently large there exists a sequence $(w_n)_{n=1}^{\infty}$ of elements in $J$ such that $\lim_{n\to\infty}S_{2h_n}(f)(w_n)=0$. That being said, consider the sequence $(z_n)_{n=1}^{\infty}$ given by Lemma \ref{lem:zeroes}. Let $n$ be large enough so that there exist numbers $0\le m,m'\le h_n$ for which $T^m(z_n),T^{-m'}(z_n)\in \tilde J$, where {$\tilde J=(a+\tfrac{b-a}{4},b-\tfrac{b-a}{4})$}.  Such $n$ exists, since due to 7) (and 6)) in Proposition \ref{prop: properties}, an interval $\tilde J$ contains at least one level of tower $\Xi_d^{(n)}$, for $n$ large enough. Moreover, by \eqref{eq: thesamelevel} as well as  2) and 4) in Proposition \ref{prop: properties}, we also have that 
\begin{equation}\label{eq: stillthesamelevel}
\begin{split}
&[T^{-h_n}(T^mz_n),T^{-h_n}(T^{-m'}z_n)]\subset J,\\
T^{2h_n}&[T^{-h_n}(T^mz_n),T^{-h_n}(T^{-m'}z_n)]
=[T^{h_n}(T^mz_n),T^{h_n}(T^{-m'}z_n)]\subset J.
\end{split}
\end{equation}

We will find for every $n\in\N$ a point $w_n \in \left(T^{-h_n}(T^mz_n),T^{-h_n}(T^{-m'}z_n)\right)$, hence, by \eqref{eq: stillthesamelevel}, we already have
\begin{equation}
    T^{2h_n}(w_n)\in J,
\end{equation}
since $T$ acts on levels of $\Xi_d^{(n)}$ as an isometry.

We are left to prove that for every $n$ large enough, we can choose 
\[w_n\in \left(T^{-h_n}(T^mz_n),T^{-h_n}(T^{-m'}z_n))\right)\] 
so that
\begin{equation}\label{eq:equalBS}
    S_{2h_n}(f)(w_n)=S_{2h_n}(f)(T^{-h_n}z_n).
\end{equation} 
Then the Property B follows from \eqref{eq: zero} in Lemma \ref{lem:zeroes}.

Since $z_n$ satisfies \eqref{eq: orbitfar} in Lemma \ref{lem:zeroes}, by 2) in Proposition \ref{prop: properties}, we have that for every $j=-2h_n,\ldots,2h_n$, the interval $(T^jz_n-3\epsilon_2|I^{(n)}|,T^jz_n+3\epsilon_2|I^{(n)}|)$ is a continuity interval of $f$ (in fact, they are all subintervals of levels of $\Xi_d^{(n)}$). 
Moreover, by Lemma \ref{lem:zeroes}, we have
\[
S_{2h_n}(T^{-h_n}(T^mz_n))-S_{2h_n}(T^{-h_n}z_n)=\sum_{i=0}^{m-1}\left(f(T^{h_n+i}z_n)-f(T^{i-h_n}z_n)\right)<0
\]
and 
\[
S_{2h_n}(T^{-h_n}(T^{-m'}z_n))-S_{2h_n}(T^{-h_n}z_n)=\sum_{i=-m'}^{-1}\left(f(T^{i-h_n}z_n)-f(T^{h_n+i}z_n)\right)>0,
\]
where the inequalities follow from the fact that $f$ is increasing and for every $i=-m',\ldots,m-1$ the points $T^{i-h_n}z_n$ and $T^{h_n+i}z_n$ are in the interval of continuity $(T^iz_n-3\epsilon_2|I^{(n)}|,T^iz_n+3\epsilon_2|I^{(n)}|)$ of $f$. Hence, by the Darboux property, there exists $w_n\in\left(T^{-h_n}(T^mz_n),T^{-h_n}(T^{-m'}z_n))\right)$ such that $S_{2h_n}(f)(w_n)=S_{2h_n}(f)(T^{-h_n}z_n)$. This shows \eqref{eq:equalBS} and finishes the proof of Property B.
    
\end{proof}

\section*{Acknowledgments}
The first author was supported by the NCN Grant No. 2025/57/B/ST1/00704. The second author was supported by the NCN Grant No. 2023/50/O/ST1/00045.

\bibliographystyle{alpha}

\bibliography{Bibliography.bib}
\end{document}